\documentclass[reqno,15pt]{amsart}
\usepackage{amsmath}
\usepackage{amssymb}
\usepackage{amscd}
\usepackage{graphics}
\usepackage{latexsym}
\usepackage{hyperref}

\theoremstyle{plain}
\newtheorem{theorem}{Theorem}[section]

\newtheorem{cor}[theorem]{Corollary}
\newtheorem{lem}[theorem]{Lemma}
\newtheorem{prop}[theorem]{Proposition}

\newcounter{kludge}
\newcounter{kludgeb}
\theoremstyle{definition}
\newtheorem{defn}[theorem]{Definition}

\newtheorem{rmk}[theorem]{Remark}

\newtheorem{notat}[theorem]{Notation}
\newtheorem{ex}[theorem]{Example}

\theoremstyle{remark}

\input xy
\xyoption{all}

\newcommand{\marpar}[1]{}
\newcommand{\mni}{\medskip\noindent}

\newcommand{\mbb}{\mathbb}
\newcommand{\QQ}{\mbb{Q}}

\newcommand{\ZZ}{\mbb{Z}}

\newcommand{\PP}{\mbb{P}}

\newcommand{\mc}{\mathcal}

\newcommand{\mcC}{\mc{C}}

\newcommand{\mcL}{\mc{L}}

\newcommand{\OO}{\mc{O}}

\newcommand{\ul}{\underline}

\newcommand{\SP}{\text{Spec }}

\newcommand{\Pic}[2]{\text{Pic}^{#1}_{#2}}

\newcommand{\Gm}[1]{\mathbb{G}_{#1}}

\newcommand{\pr}[1]{\text{pr}_{#1}}

\newcommand{\Xx}{X}

\newcommand{\FR}{K}

\newcommand{\sS}{\mathsf{S}}

\newcommand{\cb}{b}

\newsavebox{\sembox}
\newlength{\semwidth}
\newlength{\boxwidth}

\newsavebox{\semrbox}
\newlength{\semrwidth}
\newlength{\boxrwidth}

\title
{Picard Schemes of Acyclic Schemes} 

\author[Starr]{Jason Michael Starr}
\address{Department of Mathematics \\
  Stony Brook University \\ Stony Brook, NY 11794}
\email{jstarr@math.stonybrook.edu} 

\date{\today}

\begin{document}


\begin{abstract}
In his work extending rational simple connectedness to schemes with
higher Picard rank, Yi Zhu introduced hypotheses
for schemes insuring that the relative Picard functor is representable
and is \'{e}tale locally constant with finite free stalks.    
We give examples showing that one cannot eliminate any of the
hypotheses and still have a representable Picard functor that is
locally constant with finite free stalks.  We also prove that the
hypotheses are compatible with composition and with hyperplane sections.
\end{abstract}


\maketitle



\section{Acyclic Schemes} \label{sec-acyc}
\marpar{sec-acyc}

\mni
The \emph{acyclic schemes} have relative Picard functors that are
representable and that are \'{e}tale locally constant with stalks
being finite free Abelian groups.  This class includes smooth,
rationally connected varieties in characteristic $0$, as well as
mildly singular specializations of these schemes.  For this class of
schemes, the \emph{Abel maps} of \cite{dJHS} and \cite{Zhu} exist and
have good properties.  This note proves some basic properties of these
schemes.  After reviewing Zhu's theorem about the relative Picard
functor of acyclic schemes, Proposition \ref{prop-Zhu},
in the next section we present several examples showing that if any of
the hypotheses in Definition \ref{defn-Zhu} is removed, then
Proposition \ref{prop-Zhu} fails.  The compatibilities are Proposition
\ref{prop-compos}, compatibility of Definition \ref{defn-Zhu} with
composition, Corollary \ref{cor-iterate1}, compatibility of Definition
\ref{defn-Zhu} with ample hypersurfaces, Corollary \ref{cor-iterate2},
the application of Corollary \ref{cor-iterate1} to a universal family
of hypersurface sections, and Corollary \ref{cor-Lef}, the iteration
of Corollary \ref{cor-iterate2} for a universal family of complete
intersections of hypersurface sections.

\begin{defn} \label{defn-kacyclic} \marpar{defn-kacyclic}
Let $r\geq 0$ be an integer.
A projective, fppf morphism, $f:\Xx\to T$, is
\emph{$r$-acyclic for the structure sheaf} if for every $T$-scheme
$T'$ and base change morphism $f':\Xx'\to T'$, the induced morphism
$\OO_{T'} \to Rf'_*\OO_{\Xx'}$ is a quasi-isomorphism in all degrees
$\leq r$.  The morphism is \emph{$\OO$-acyclic} if it is $r$-acyclic
for every $r\geq 0$.
\end{defn}

\begin{rmk} \label{rmk-kacyclic} \marpar{rmk-kacyclic}
Every projective, fppf morphism is locally on the target the base change
of a projective, fppf morphism of Noetherian schemes.  If $T$ is
Noetherian, then $f$ is $r$-acyclic if and only if for every geometric
point $t$ of $T$, $\kappa(t)\to H^0(X_t,\OO_{X_t})$ is an isomorphism
and $h^q(X_t,\OO_{X_t})$ equals $0$ for every $0<q\leq r$
by Cohomology and Base Change, \cite[Theorem
III.12.11]{H}.
\end{rmk}

\begin{defn}\cite[Definition 2.10]{Zhu} \label{defn-Zhu} \marpar{defn-Zhu}
A projective morphism $f:X_T\to T$
is \emph{weakly acyclic},
resp. \emph{acyclic}, if 
\begin{enumerate}
\item[(i)] $f$ is fppf,
\item[(ii)] every $X_t$ is LCI and
  $\text{codim}_{X_t}(\text{Sing}(X_t))$ is $\geq 3$, resp. is $\geq 4$,
\item[(iii)] $X_t$ is $2$-acyclic for the structure sheaf, and
\item[(iv)] $X_t$ is algebraically simply connected. 
\end{enumerate}
The \emph{acyclic locus}, resp. \emph{weakly acyclic locus}, 
is the maximal open subscheme $T^o\subset T$
such that $X_T\times_T T^o\to T^o$ is acyclic, resp. weakly acyclic.
\end{defn}

\begin{prop}\cite[Proposition X.1.2]{SGA1} \label{lem-Zhu} \marpar{lem-Zhu}
For every proper, fppf morphism $f:X_T\to T$ whose geometric fibers
are reduced, the finite part of the Stein factorization is \'{e}tale
over $T$.
\end{prop}

\begin{proof}
By limit theorems, it suffices to prove the result when $T$ is a
Noetherian scheme (even finitely presented over $\SP \ZZ$).  The
finite part of the Stein
factorization is a finite morphism to $T$.  To prove that it is
\'{e}tale, it suffices to prove that it is formally \'{e}tale, e.g.,
it suffices to prove that it is formally \'{e}tale after base change
to the strictly Henselized local ring $\OO_{T,t}^{sh}$ for every $t$ in $T$.
The
Stein factorization is compatible with flat base change of $T$.  Thus,
without loss of generality, assume that $T$ equals $\SP \OO_{T,t}^{sh}$.
By \cite[Proposition 18.5.19]{EGA4}, it suffices to consider the case
that $X_t$ is connected.  Since $X_t$ is connected, projective, and
reduced over the algebraically closed field $\kappa(t)$, the natural
homomorphism $\kappa(t) \to H^0(X_t,\OO_{X_t})$ is an isomorphism.
Thus, the composition,
$$
\kappa(t) \xrightarrow{f^\#_t} f_*\OO_{X_T}\otimes_{\OO_T} \kappa(t)
\to H^0(X_t,\OO_{X_t}),
$$
is an isomorphism.  By Cohomology and Base Change, cf. \cite[Theorem
III.12.11]{H}, the following natural homomorphism is an isomorphism,
$$
f^\#:\OO_T \to f_*\OO_{X_T}.
$$
Thus, the Stein factorization is an isomorphism, hence it is formally
\'{e}tale.  
\end{proof}

\begin{prop}\cite[Proposition 2.9]{Zhu} \label{prop-Zhu}
  \marpar{prop-Zhu}
For every weakly acyclic morphism, and even for morphisms that become
weakly acyclic after base change by an \'{e}tale cover of $T$, 
the relative Picard functor of $X_T/T$
is representable, and it
is \'{e}tale locally constant with finite free stalks.
\end{prop}

\begin{proof} 
This is a review of the proof in \cite{Zhu}.  By limit theorems, it
suffices to prove the result when $T$ is a Noetherian scheme (even finitely
presented over $\SP \ZZ$) and $f$ is weakly acyclic.  By Proposition
\ref{lem-Zhu}, the finite part $T'$ of the Stein factorization of $f$ is
finite and \'{e}tale over $T$.  The relative Picard functor of $X_T/T$
is the restriction of scalars relative to $T'/T$ of the relative
Picard functor of $X_T/T'$.  Thus, it suffices to prove the result for
$X_T/T'$.  Thus, without loss of generality, assume that the geometric
fibers of $f$ are connected.  By Hypothesis (ii), the geometric fibers
are integral.

\mni
Because $f$ is
projective and flat with integral geometric fibers, 
the relative Picard functor is
representable and equals a union of open and closed
subschemes that are quasi-projective over $T$, 
cf. \cite[Theorem 3.1, no. 232-06]{FGA}.  
By Hypothesis (iii), the
relative Picard functor is formally unramified and formally smooth
over $T$.  Thus, it is formally \'{e}tale over $T$.  Since the Picard functor
is representable and locally finitely presented over $T$, it is
\'{e}tale over $T$.

\mni
Since the open and closed quasi-projective schemes are \'{e}tale over
$T$, they are finite over $T$ if and only if they are proper over $T$.
To prove properness,
it suffices to verify the valuative criterion of
properness.  Thus, assume that $T$ is $\SP \OO_T$ for a DVR $\OO_T$.  Let
$\mcL_\eta$ be an invertible sheaf on the generic fiber $X_\eta$ of
$f$.  
Denote by
$X_{T,\text{sm}}\subset X_T$ the smooth locus of $f$.  Since $X_\eta$
is projective, $\mcL_\eta$ comes from a Cartier divisor $D$ on $X_\eta$.
Since $X_{T,\text{sm}}$ is regular, the Cartier divisor $D$ on
$X_\eta\cap X_{T,\text{sm}}$ extends to a Cartier divisor on
$X_{T,\text{sm}}$.  Thus, the invertible sheaf extends to an
invertible sheaf on the open subscheme $U=X_\eta \cup
X_{T,\text{sm}}$.  Since $X_T$ is normal, the pushforward of this
invertible sheaf
from $U$ is a torsion-free coherent sheaf $\mcL$ that is $\sS_2$.  
Denote by
$V\subset X_T$ the maximal open subscheme on which $\mcL$ has rank
$\leq 1$.  

\mni
For a generic point $x\in X_T$ of the complement of $V$,
the stalk of $\mcL$ at $x$ has rank $\geq 2$.  
Since $x$ is
in the closed fiber and in the complement of the smooth locus, $x$ has
codimension $\geq 3$ in the closed fiber, hence codimension $\geq 4$
in $X_T$.  Since $X_T$ is a local complete intersection scheme, by 
\cite[Th\'{e}or\`{e}me XI.3.13]{SGA2}, the local ring
$\OO_{X_T,x}$ is parafactorial.  Thus, the stalk of $\mcL$ at $x$ is
locally free of rank $1$.  This contradiction proves that $V$ is all
of $X_T$, i.e., $\mcL$ is an invertible sheaf on $X_T$.  Therefore, by
the valuative criterion of properness, for every Noetherian scheme $T$
and for every fppf projective morphism $f:X_T\to T$ satisfying
Hypotheses (i)-(iv), the relative Picard function is representable and
equals a union of open and closed subschemes, each of which is finite,
\'{e}tale over $T$.

\mni
For every point $t$ of $T$, the geometric Picard group of $X_t$ is
finitely generated 
by the theorem of the base, \cite[Th\'{e}or\`{e}me XIII.5.1]{SGA6}.
By Hypothesis (iv), the geometric Picard group is torsion-free.  Thus,
it is finite free of some rank $r\geq 1$.  Define $T_{r}$,
resp. $T_{\geq r}$, 
to be the subset of $T$ over which the geometric
Picard group is finite free of rank $r$, resp. of rank $\geq r$.

\mni
Let $r_0$ be an integer, and let $t\in T_{\geq r_0}$ be a point of
rank $r\geq r_0$.
The \'{e}tale stalk at $t$ of the Picard functor 
is generated by the images of
finitely many of the finite, \'{e}tale, open and closed subschemes of
the relative Picard scheme.  The image under $f$ of each of these is
an open and closed subscheme of $T$ that contains $t$.  The
intersection of these finitely many open and closed subscheme of $T$
is an open and closed subscheme of $T$ that contains $t$.  For every
geometric point of this open and closed subscheme, the rank is $\geq
r$.  In particular, the rank is $\geq r_0$.  
Thus, each subset $T_{\geq r_0}\subset T$ is open, and it is a union
of open subsets that are both open and closed.  

\mni
Since $T$ is
Noetherian, there are only finitely many irreducible components.
Thus, there are 
also finitely many connected components.  The subset $T_r$
contains the unique connected component of $T$ that contains $t$.
Thus, also every subset $T_r$ is an open subset of $T$.  The
restriction of the relative Picard functor over $T_r$ is \'{e}tale
locally constant with finite free stalks of rank $r$.
\end{proof}

\mni
\textbf{Acknowledgments.}  This is part of a project begun with
Chenyang Xu for extending theorems about rational simple
connectedness; I am grateful to Xu for all his help.  I am also
grateful to Yi Zhu for many discussions about his work.  I am grateful
to Aise Johan de Jong for help with references.  I was supported
by NSF Grants DMS-0846972 and DMS-1405709, as well as a Simons
Foundation Fellowship.


\section{Examples and Composition} \label{sec-compos}
\marpar{sec-compos} 

\begin{ex} \label{ex-0} \marpar{ex-0}
Let $Q\subset \PP^3$ be a smooth quadric surface.
For $T$ equal to $\mathbb{A}^1$, for $X_T$ the reduced closed
subscheme of $T\times \PP^3$ whose intersection with $\Gm{m}\times
\PP^3$ equals $\Gm{m}\times Q$ and whose fiber over $0\in T$ equals
all of $\PP^3$, then $f$ satisfies Hypotheses (ii), (iii), and (iv),
yet the morphism is not flat.  The relative Picard functor is
representable and \'{e}tale over $T$, but it fails the valuative
criterion of properness.
\end{ex}

\begin{ex} \label{ex-1} \marpar{ex-1}
For every integer $r\geq 2$, for
$T$ equal to $\mathbb{A}^1$, and for $X_T$ a specialization of the
image of the Segre embedding, $\sigma:\PP^r\times \PP^r\to \PP^{r^2+2r},$ to a
cone over a smooth hyperplane section of $\sigma(\PP^2\times \PP^2)$,
Hypotheses (i), (iii), and (iv) are
satisfied, and the fibers are smooth in codimension $\leq 2$, yet the
fibers are not local complete intersections, and the relative Picard
scheme is not \'{e}tale locally constant.  More precisely, the
relative Picard scheme is separated and \'{e}tale over $T$, 
but it fails the valuative
criterion of properness.  
\end{ex}

\begin{ex} \label{ex-2} \marpar{ex-2}
For $T$ equal to $\mathbb{A}^1$ and for $X_T$ a specialization in
$\PP^3$ of a smooth quadric hypersurface to a quadric hypersurface
with an ordinary double point, Hypothesis (i), (iii), and (iv) are
satisfied, and the fibers are local complete intersections, yet the
special fiber is singular at a point of codimension $2$.  The relative
Picard scheme is not proper over $T$.
\end{ex}

\begin{ex} \label{ex-3} \marpar{ex-3}
For a family of 
supersingular Enriques surfaces over a smooth scheme
$T$ in characteristic $2$, Hypotheses (i), (ii),
and (iv) are satisfied, yet Hypothesis (iii) fails.  The relative
Picard functor is representable and \'{e}tale locally constant over $T$.  Yet
the relative Picard functor is not smooth over $T$: the connected
component of the identity is $\alpha_2$.  
\end{ex}

\begin{ex} \label{ex-4} \marpar{ex-4}
For a family of Enriques surfaces over a smooth scheme $T$ in
characteristic $0$, Hypotheses (i), (ii), and (iii) are satisfied, yet
Hypothesis (iv) fails.  The relative Picard functor is representable
and \'{e}tale locally constant over $T$.  Yet the stalks have
$\ZZ/2\ZZ$-torsion.  
\end{ex}

\begin{prop} \label{prop-compos} \marpar{prop-compos}
Let $g:Y\to X$ and $f:X\to T$ be projective, fppf morphisms whose
geometric fibers are integral.  The composition $f\circ g$ is a
projective, fppf morphism whose geometric fibers are integral.
If both $g:Y\to X$ and $f:X\to T$ are $r$-acyclic, resp.
acyclic, weakly acyclic,
then so is the composition $f\circ g:Y\to T$.
\end{prop}

\begin{proof}
By limit theorems, it suffices to prove the case when $T$ is
Noetherian.  A composition of projective, fppf morphisms is a
projective, fppf morphism.  For each geometric point $t$ of $T$, the
fiber $X_t$ of $f$ is integral.  Denote by $\eta$ the generic point.  
The morphism $g_t:Y_t\to X_t$ is
projective and flat.  Thus, for every nonempty open affine $U\subset
Y_t$, $U$ intersects the generic fiber $Y_{t,\eta}=g_t^{-1}(\eta)$.  
Since $\OO_{Y_t}(U)$ is
$\OO_T$-flat, the induced morphism $\OO_{Y_t}(U)\to
\OO_{Y_{t,\eta}}(U\cap Y_{t,\eta})$ is injective.  Since the geometric
fibers of $g$ are integral, the fiber $Y_{t,\eta}$ is integral.  Since
$\OO_{Y_t}(U)$ is a subring of an integral domain, also $\OO_{Y_t}(U)$
is an integral domain.  Therefore $Y_t$ is integral.  So the geometric
fibers of $f\circ g$ are integral.

\mni
A composition of flat, LCI morphisms is a flat, LCI morphism, cf. the
proof of \cite[Proposition 6.6(c)]{F} (Fulton works with global
embeddings in smooth schemes, but the diagram in the proof also proves
the result in the local case).  With notation as in the previous
paragraph, if $\text{Sing}(X_t)$ has codimension $\geq c$ in $X_t$,
then also $g_t^{-1}(\text{Sing}(X_t))$ has codimension $\geq c$ in
$Y_t$, since $g_t$ is flat.  If the singular locus of the morphism
$g_t$ has codimension $\geq c$ in every fiber of $g_t$, then it has
codimension $\geq c$ in $Y_t$.  Then the union of the singular locus
of $g_t$ and $g^{-1}(\text{Sing}(X_t))$ has codimension $\geq c$ in
$Y_t$.  On the open complement of this union, $f\circ g$ is a
composition of smooth morphisms, hence it is smooth.  Thus, the
singular locus of $Y_t$ is contained in this union, so that the
singular locus of $Y_t$ has codimension $\geq c$ in $Y_t$.  Finally,
if the geometric fibers of $g_t$ are algebraically simply connected,
and if $X_t$ is algebraically simply connected, then also $Y_t$ is
algebraically simply connected, cf. \cite[Corollaire IX.6.11]{SGA1}.

\mni
Thus, to prove that $f\circ g$ is acyclic, resp. weakly acyclic, it
suffices to prove that it is $2$-acyclic for the structure sheaf.  For
projective, fppf morphisms $f$ and $g$ that are $r$-acyclic, consider
the Leray spectral sequence,
$$
E^{p,q}_2 = H^p(X_t,R^q(g_t)_* \OO_{Y_t}) \Rightarrow H^{p+q}(Y_t,\OO_{Y_t}).
$$
Since $g$ is $r$-acyclic, and since $g_t$ is a base change of $g$,
also $g_t$ is $r$-acyclic.  Thus, $(g_t)_*\OO_{Y_t}$ equals
$\OO_{X_t}$, and $R^q(g_t)_*\OO_{Y_t}$ is the zero sheaf for $0<q\leq
r$.  Thus, for every integer $s$ with $0\leq s \leq r$, the only
nonzero terms in the spectral sequence with $p+q=s$ are when $q$
equals $0$ and $p$ equals $s$, i.e., $E^{s,0}_2=H^s(X_t,\OO_{X_t})$.
Since $f$ is $r$-acyclic, this equals $0$ unless $s=0$, in which case
it equals $H^0(X_t,\OO_{X_t}) = \kappa(t)$.  Thus,
$H^s(Y_t,\OO_{Y_t})$ equals $0$ for $0<s\leq r$, and the natural map
$\kappa(t) \to H^0(Y_t,\OO_{Y_t})$ is an isomorphism.  So $f\circ g$
is also $r$-acyclic for the structure sheaf.
\end{proof}


\section{Hyperplane Theorems} \label{sec-hyperplane}
\marpar{sec-hyperplane} 

\mni
The following lemma in characteristic zero follows by the
Kawamata-Viehweg Vanishing Theorem.

\begin{lem} \label{lem-KV} \marpar{lem-KV}
Let $\FR$ be a field.  
Let $f:\Xx\to T$ be a proper, fppf morphism of finite type
$\FR$-schemes of relative dimension $n$. 
Let $Y\subset \Xx$ be an effective Cartier divisor that
is $T$-flat and $f$-ample.  
If $\Xx$ is smooth over $\FR$, then $T$ is smooth over $\FR$, and
every fiber of $f$ is LCI.  If, moreover, $\text{char}(\FR)$ equals
$0$, and 
if $n\geq r+2$, then
$H^q(\Xx_t,\OO_{\Xx_t}(-\ul{Y}_t))$ is zero for every
$q=0,\dots,r+1$.  
Thus, if $f$ is $r$-acyclic for the
structure sheaf,
then also $f|_Y:Y\to
T$ is $r$-acyclic for the structure sheaf.
\end{lem}

\begin{proof}
Since $\Xx$ is $\FR$-smooth and since $f$ is flat, also $T$ is
$\FR$-smooth, 
\cite[Proposition 17.7.7]{EGA4}.  For a flat morphism from an LCI
scheme to a regular scheme, every fiber is LCI.  In particular, every
fiber is Gorenstein.

\mni  
The relative dualizing sheaf of $f$ is 
$$
\omega_{\Xx/T} \cong \omega_{\Xx/\FR}\otimes_{\OO_\Xx} f^*
\omega_{T/\FR}^\vee.  
$$
The dualizing sheaf of each fiber is the
restriction of $\omega_{\Xx/T}$.     

\mni
Now assume that $\text{char}(\FR)$ equals $0$, and assume that $\Xx$
is smooth over $\FR$.
By the Kawamata-Viehweg Vanishing Theorem, \cite[Theorem
1.2.3, p. 306]{KaMaMa}, for every $q>0$, $R^qf_*
\omega_{\Xx/T}(\ul{Y})$ is zero.  Thus, for every geometric point $t$
of $T$, for every $q>0$, $H^q(\Xx_t,\omega_{\Xx_t}(\ul{Y}_t))$ is
zero by Cohomology and Base Change, 
\cite[Theorem III.12.11]{H}.  By Serre duality, also
$H^q(\Xx_t, \OO_{\Xx_t}(-\ul{Y}_t)$ is zero for every $q < n$.  

\mni
For the short exact sequence
$$
0\to \OO_{\Xx_t}(-\ul{Y}_t) \to \OO_{\Xx_t} \to \OO_{Y_t} \to 0
$$
the long exact sequence of cohomology gives
$$
H^q(\Xx_t,\OO_{\Xx_t}(-\ul{Y}_t)) \to H^q(\Xx_t,\OO_{\Xx_t}) \to
H^q(Y_t,\OO_{Y_t}) \to H^{q-1}(\Xx_t,\OO_{\Xx_t}(-\ul{Y}_t)).
$$
Thus, for every $q\leq n-2$, the restriction map is an isomorphism,
$$
H^q(\Xx_t,\OO_{\Xx_t}) \xrightarrow{\cong} H^q(Y_t,\OO_{Y_t}).
$$
Since $r\leq n-2$, also $H^q(Y_t,\OO_{Y_t})$ is zero for
$q=1,\dots,r$.  Also, the composition
$$
\OO_T\otimes_{\OO_T} \kappa(t) \to f_*\OO_{\Xx}\otimes_{\OO_T}
\kappa(t) 
\to f_*\OO_Y\otimes_{\OO_T}\kappa(t) \to H^0(Y_t,\OO_{Y_t}) 
$$
is an isomorphism.  Thus, once again using Cohomology and Base Change,
for arbitrary $T'$, also $R^qf'_*\OO_{Y'}$ is zero for $q=1,\dots,r$,
the natural map $\OO_{T'} \to f'_*\OO_{Y'}$ is an isomorphism.
\end{proof}

\begin{prop}\cite{SGA2} \label{prop-iterate1} \marpar{prop-iterate1}
Let $f:\Xx\to T$ be a proper, fppf morphism of Noetherian schemes of
pure relative dimension $n$.
Let $Y\subset \Xx$ be an effective Cartier divisor that is $T$-flat
and $f$-ample.  For every geometric point $t$ of $T$, denote $X_t$,
resp. $Y_t$, the corresponding fiber of $X$, resp. $Y$.
\begin{enumerate}
\item[(i)] If $n\geq 2$, if $X_t$ is integral and satisfies Serre's
  condition $\sS_3$,
and if $\text{codim}_{Y_t}(\text{Sing}(Y_t)) \geq 2$, then $Y_t$ is
integral and normal.
\item[(ii)] If $n\geq 3$ and if $X_t$ is LCI with
  $\text{codim}_{X_t}(\text{Sing}(X_t)) \geq 3$, then
  $\pi_1^{\text{alg}}(Y_t)\to \pi_1^{\text{alg}}(X_t)$ is an
  isomorphism.
\item[(iii)] If $n\geq 4$, if $T$ is a finite type scheme over a
  characteristic $0$ 
  field $\FR$, if $\Xx$ is smooth over $\FR$, and if
  $\text{codim}_{X_t}(\text{Sing}(X_t)) \geq 4$, then
  $\text{Pic}(X_t)\to \text{Pic}(Y_t)$ is an isomorphism.  
\end{enumerate}
\end{prop}

\begin{proof}
\textbf{(i)}  Since $\Xx_t$ satisfies $\sS_3$, also $Y_t$ satisfies
$\sS_2$.  Since $Y_t$ is regular at every codimension $0$ and
codimension $1$ point, $Y_t$ is normal by Serre's Criterion
\cite[Th\'{e}or\`{e}me 5.8.6]{EGA4}.
Finally, by  \cite[Corollaire XII.3.5]{SGA2}, $Y_t$ is
connected.  Thus $Y_t$ is integral.

\mni
\textbf{(ii)}
By
the Purity Theorem, \cite[Th\'{e}or\`{e}me X.3.4(ii)]{SGA2}, $\Xx_t$
is pure and of depth $\geq 3$
at every closed point.  By the
Lefschetz Hyperplane Theorem for
\'{e}tale fundamental groups, \cite[Corollaire XII.3.5]{SGA2}, the
natural homomorphism $\pi_1^{\text{alg}}(Y_t) \to
\pi_1^{\text{alg}}(\Xx_T)$ is an isomorphism.

\mni
\textbf{(iii)}
By
Lemma \ref{lem-KV}, $H^q(\Xx_t,\OO_{\Xx_t}(-d\ul{Y}_t)$ is zero for
all $d>0$ and $q=1,2$.  By Grothendieck's proof of Samuel's
Conjecture, \cite[Th\'{e}or\`{e}me XI.3.13(ii),
Corollaire XI.3.14]{SGA2}, 
the scheme $\Xx_t$ is parafactorial, and even factorial.
By the Lefschetz Hyperplane Theorem for Picard groups,
\cite[Corollaire XII.3.6]{SGA2}, the restriction on Picard
groups is an isomorphism.
\end{proof}

\begin{cor} \label{cor-iterate1} \marpar{cor-iterate1}
Let $\FR$ be a characteristic $0$ field, and let $f:\Xx\to T$ be a
proper, fppf morphism of $\FR$-schemes of pure dimension $n$.  
Let $Y\subset \Xx$ be an effective Cartier divisor that is
$T$-flat and $f$-ample.  If $n\geq 4$, if $\Xx$ is smooth over $\FR$,
if $\text{codim}_{Y_t}(\text{Sing}(Y_t)) \geq 4$ for every geometric point
$t$ of $T$,
and if $f$ is acyclic, then also $f|_Y:Y\to T$
is acyclic.  Moreover, the restriction morphism of
\'{e}tale group schemes, $\text{Pic}_{X/T}\to \text{Pic}_{Y/T}$, is an
isomorphism.  
\end{cor}

\begin{proof}
By Proposition \ref{lem-Zhu}, the finite part of the Stein factorization of
$f$ is finite and \'{e}tale over $T$.  Up to replacing $T$ by this
finite, \'{e}tale cover, assume that $f$ has integral geometric
fibers.

\mni
By hypothesis, $f|_Y:Y\to T$ is flat.  By Proposition
\ref{prop-iterate1}(i), the geometric fibers are integral.  
By Lemma
\ref{lem-KV} and by Proposition \ref{prop-iterate1}(ii), Definition
\ref{defn-Zhu}(ii) and (iv) hold.  
By Lemma \ref{lem-KV}, Definition
\ref{defn-Zhu}(iii) holds.  Finally, by Proposition
\ref{prop-iterate1}(iii), the restriction morphism of Picard schemes is
an isomorphism on geometric fibers.  Since this is a morphism of
\'{e}tale $T$-schemes, the restriction morphism is \'{e}tale.  Since
it is also bijective on geometric points, it is an isomorphism.    
\end{proof}


\section{Families of Hypersurfaces} \label{sec-Bertini}
\marpar{sec-Bertini} 

\mni
Let $\Xx\to T$, $\mcC\to T$, and $\mcC\to G$ be fppf morphisms.

\begin{lem} \label{lem-Bertini} \marpar{lem-Bertini}
Assume that the schemes above are finite type over a field $\FR$, and
assume that the morphisms are $\FR$-morphisms.
If $\Xx$ is smooth over $\FR$, and if $\mcC\to T$ is smooth, then also
$\Xx\times_T \mcC$ is smooth over $\FR$.  If $\text{char}(\FR)$ equals
$0$, then there exists a dense open subset $W\subset G$ such that the
morphism $\Xx\times_T \mcC \times_G W\to W$ is smooth.  
\end{lem}

\begin{proof}
Since $\mcC\to T$ is smooth, also $\Xx\times_T \mcC \to \Xx$ is
smooth.  Since $\Xx$ is smooth over $\FR$, also $\Xx\times_T \mcC$ is
smooth over $\FR$.  If $\text{char}(\FR)$ equals $0$, then by the
Generic Smoothness Theorem,
cf. \cite[Corollary III.10.7]{H}, there exists a dense open subset
$W\subset G$ such that $\Xx\times_T \mcC \times_G W\to W$ is smooth.
\end{proof}

\begin{notat} \label{notat-T} \marpar{notat-T}
Let $T$ be a Noetherian scheme of pure dimension $m$.  Let $\Xx_T\subset
\PP^r_T$ be a closed subscheme such that $p:\Xx_T \to T$ is flat
of pure
relative dimension $n\geq 1$.  Denote by $\Xx_T^{\text{sm}}\subset \Xx_T$ 
the open subscheme
on which $p$ is smooth.  
\end{notat}

\mni
By \cite[Expos\'{e} XV, Corollaire 1.3.4]{SGA7II},
there exists an open subscheme $\Xx_T^{\text{odp}}\subset
\Xx_T$ consisting of points of geometric fibers where either $p$ is
smooth or else has an
ordinary double point.  

\begin{defn} \label{defn-Bertini} \marpar{defn-Bertini}
The \emph{smooth locus of $p$}, 
$T^{\text{sm}}$, is the open complement in $T$ 
of $p(X\setminus \Xx_T^{\text{sm}})$. 
Similarly, the \emph{ordinary locus of $p$},
$T^{\text{odp}}\subset T$, is the open complement of
$p(X\setminus\Xx_T^{\text{odp}})$, i.e., the maximal open subscheme of
$T$ over which $p$ has geometrically reduced fibers that are either
smooth or else admit (at worst) finitely many ordinary double points.
Over the open $T^{\text{odp}}$, the morphism $X\setminus
\Xx_T^{\text{sm}} \to T$ is finite.  The \emph{Lefschetz
  locus}, 
$T^{\text{Lef}}$, is the
maximal open subscheme of $T^{\text{odp}}$ over which this finite morphism is
a closed immersion.  Thus, over $T^{\text{Lef}}$, every geometric
fiber is either smooth or else it is reduced with a single ordinary
double point.  
\end{defn}

\mni
Denote by $P_r(t)\in \QQ[t]$ the numerical polynomial such that
$P_r(s)$ equals $\binom{r+s}{r}$ for every integer $s\geq -r$.  
For each integer $d\geq 1$, 
the projective space $\PP^{N_d}_T
= \text{Hilb}^{P_{r}(t)-P_r(t-d)}_{\PP^r_T/T}$ parameterizes degree
$d$ hypersurfaces $H\subset \PP^r$.

\begin{defn} \label{defn-dual} \marpar{defn-dual}
The \emph{degenerate locus} or \emph{dual locus}, 
$\check{\Xx}_T$,  is
the closed
subset of $\PP^{N_d}_T$ whose geometric points relative to
$\SP \kappa \to T$, 
parameterize hypersurfaces
$H\subset \PP^r_\kappa$ for which $H\cap \Xx_\kappa$ is \textbf{not} a smooth
$\kappa$-scheme of dimension $n-1$, i.e., either it has an irreducible
component of dimension $\geq n$ or else it is singular.  
The \emph{badly degenerate locus},
$F_1\subset \check{\Xx}_T$, is 
the closed subset such that $H\cap
\Xx_\kappa$ either (i) has an irreducible component of dimension $\geq n$,
(ii) it is nonreduced, or (iii) 
it is reduced of dimension $n$, yet it has worse
than a single ordinary double point singularity.  
\end{defn}

\mni
For the universal
family of hypersurface sections of $\Xx_T$ over $\PP^{N_d}_T$, say $Y\to
\check{\PP}^r_T$, the degenerate locus, resp. the badly degenerate
locus,
is the union of the
non-flat locus with the closed complement of
$(\check{\PP}^r_T)^{\text{sm}}$,
resp. $(\check{\PP}^r_T)^{\text{Lef}}$, as defined in Definition
\ref{defn-Bertini}.  Thus, the degenerate locus and the badly
degenerate locus are closed subsets.

\begin{cor} \label{cor-iterate2} \marpar{cor-iterate2}
Let $\FR$ be a field.
With notations as above, assume that $T$ is a finite type
$\FR$-scheme, and assume that $\Xx_T$ is smooth over $\FR$.  Then
$\Xx_{\PP^{N_d}} := \Xx_T\times_T \PP^{N_d}_T$ is smooth over
$\FR$. 
Also the universal hypersurface, $Y\subset \Xx_{\PP^{N_d}}$ as
above, is smooth over $\FR$.
If $\text{char}(K)$ equals $0$, if $\Xx_T/T$ is acyclic,
and if $n\geq 4$, resp. if $n\geq 5$, then the restriction
of $Y$ over $\PP^{N_d}_T\setminus \check{\Xx}_T$, resp. over
$\PP^{N_d}_t\setminus F_1$, is acyclic. Also over this
(respective) open subset,
the natural
morphism from the pullback of $\Pic{}{\Xx_T/T}$ to the relative Picard
scheme of $Y$ is an isomorphism.
\end{cor}

\begin{proof}
By Lemma \ref{lem-Bertini}, $\Xx_{\PP^{N_d}}$ is smooth over $\FR$.
The same method proves that $Y$ is smooth over $\FR$: the projection
$Y\to \Xx_{\PP^{N_d}}$ is a projective space bundle.  If $n\geq 4$,
then the hypotheses of Corollary \ref{cor-iterate1} are satisfied for
$Y\to \PP^{N_d}_T$ over $\PP^{N_d}_T\setminus \check{\Xx}_T$.  If
$n\geq 5$, then over $\PP^{N_d}_T\setminus F_1$, the fibers of $Y_t$
have singular locus of codimension $n-1\geq 4$, so the hypotheses are
satisfied over $\PP^{N_d}_T\setminus F_1$.  
\end{proof}

\begin{prop}\cite[Expos\'{e} XVII, Th\'{e}or\`{e}me
  2.5]{SGA7II} \label{prop-Lef} \marpar{prop-Lef}  
Assume that $T^{\text{Lef}}$ equals all of $T$, and assume that
$T^{\text{sm}}$ is a dense open subset of $T$.  Then for every $d\geq
2$, 
every irreducible
component of $\check{\Xx}_T$, resp. of $F_1$, has codimension $\geq 1$,
resp. $\geq 2$, in $\PP^{N_d}_T$.  In characteristic $0$ this also holds
with $d=1$.
\end{prop}

\begin{proof}
The statement over $T^{\text{sm}}$ follows directly from loc. cit.  By
hypothesis, every component of the singular locus, 
$\Delta:=T\setminus T^{\text{sm}}$, has
codimension $\geq 1$ in $T$.  The inverse image of $\Delta$ in
$\PP^{N_d}_T$ has codimension $1$.  For each geometric point $\SP
\kappa \to \Delta$, since this is a point of $T^{\text{Lef}}$, the
corresponding fiber $\Xx_\kappa$ has a single ordinary double point
$x$.  Inside $\PP^{N_d}_\kappa$, the set parameterizing $H$ with $x\in
H$ is a proper closed subset, hence has codimension $\geq 1$.  In
total, the locus in $\Delta \times_T \PP^{N_d}_T$ parameterizing $H$
containing a singular point of $p$ is a subset of codimension $\geq
2$.  Thus the proposition over all of $T$ is reduced to the
proposition over $T^{\text{sm}}$.
\end{proof}


\section{Families of Complete Intersections} \label{sec-CI}
\marpar{sec-CI} 

\mni
Let $\Xx_T \to T$ be an fppf morphism of pure relative dimension $n$.

\begin{notat} \label{notat-CI} \marpar{notat-CI} 
Let $\cb$ be an integer with $1\leq
\cb\leq n$, let
$(\iota_j:\Xx_T \hookrightarrow \PP^{r_j}_T)_{1\leq j\leq \cb}$ be an
ordered $\cb$-tuple of closed immersions with associated very ample
invertible sheaves $\mc{A}_j = \iota_j^*\OO_{\PP^{r_j}_T}(1)$.  Let
$\underline{d}=(d_1,\dots,d_\cb)$ be an ordered $\cb$-tuple 
of integers $d_i\geq 1$.  For each
$j=1,\dots,\cb$, denote by $V_j(d_j)$ the free $\OO_T$-module
$H^0(\PP^{r_j}_T,\OO_{\PP^{r_j}_T}(d_j))$.  Denote by
$V(\underline{d})$ the direct sum $V_1(d_1)\oplus \dots \oplus
V_\cb(d_\cb)$ as a free $\OO_T$-module.  Denote by $\PP_T
V(\underline{d})$ the projective space over $T$ on which there is a
universal ordered
$\cb$-tuple $(\phi_1,\dots,\phi_\cb)$ 
of sections of the invertible sheaves $\OO_{\PP^{r_j}}(d_j)$. 
Precisely, for
the product 
$$
P=\PP_T V(\underline{d}) \times_T (\PP_T^{r_1}\times \dots
\PP_T^{r_\cb})
$$
with its projections 
$$
\text{pr}_0:P\to
\PP_TV(\underline{d}) \text{ and }
\text{pr}_j:P\to \PP_T^{r_j},
$$
the sequence
$(\phi_1,\dots,\phi_\cb)$ is a universal homomorphism of coherent sheaves
$$
\text{pr}_1^*\OO_{\PP^{r_1}_T}(-d_1)\oplus \dots \oplus
\text{pr}_\cb^*\OO_{\PP^{r_\cb}_T}(-d_\cb) \to \text{pr}_0^*\OO_{\PP_T
  V(\underline{d})}(1),  
$$
or equivalently, a universal homomorphism of coherent sheaves,
$$
(\phi_1,\dots,\phi_\cb):\text{pr}_0^*\OO_{\PP_T
  V(\underline{d})}(-1)\otimes\left(\text{pr}_1^*\OO_{\PP^{r_1}_T}(-d_1)\oplus
  \dots \oplus 
\text{pr}_\cb^*\OO_{\PP^{r_\cb}_T}(-d_\cb)\right) \to \OO_P.  
$$
\end{notat}

\mni
For the diagonal closed immersion $\iota = (\iota_1,\dots,\iota_\cb)$ of
$\Xx_T$ into $\PP_T^{r_1}\times_T \dots \times_T \PP_T^{r_\cb}$,  for
every  $j=1,\dots,c$, there is an associated homomorphism of coherent
sheaves on $\PP_T V(\underline{d})\times_T \Xx_X$,
$$
\iota^*\phi_j:\text{pr}_0^*\OO_{\PP_T
  V(\underline{d})}(-1)\otimes_{\OO}
\text{pr}_1^*\iota_j^*\OO_{\PP^{r_j}_T}(-d_j) \to \OO_{\PP_T
  V(\underline{d}) \times_T \Xx_T}.
$$

\begin{defn} \label{defn-CI} \marpar{defn-CI}
Define $Y_j$ to be the Cartier divisor
on $\PP_T V(\underline{d})\times_T \Xx_T$
whose ideal sheaf is the image of $\iota^*\phi_j$.
For every $j=0,\dots,\cb$, define the
closed subscheme
$\Xx_j \subset \PP_T V(\underline{d})\times_T \Xx_T$ recursively by
$$
\Xx_0 = \PP_T V(\underline{d}) \times_T \Xx_T \text{ and } \Xx_j = Y_j \cap
\Xx_{j-1}
$$ 
for every $j=1,\dots,\cb$.  
Define two sequences of open subsets 
$$
\PP_T V(\underline{d})^{\text{sm}}_\cb \subset \PP_T
V(\underline{d})^{\text{sm}}_{\cb-1}\subset \dots 
\subset \PP_T V(\underline{d})^{\text{sm}}_2 \subset \PP_T
V(\underline{d})^{\text{sm}}_1 \subset 
\PP_T V(\underline{d})^{\text{sm}}_0 = \PP_T V(\underline{d}),  
$$
respectively,
$$
\PP_T V(\underline{d})^{\text{Lef}}_\cb \subset \PP_T
V(\underline{d})^{\text{Lef}}_{\cb-1}\subset \dots 
\subset \PP_T V(\underline{d})^{\text{Lef}}_2 \subset \PP_T
V(\underline{d})^{\text{Lef}}_1 \subset 
\PP_T V(\underline{d})^{\text{Lef}}_0 = \PP_T V(\underline{d}), 
$$
where for $i=1,\dots,\cb$, 
$\PP_T V(\underline{d})^{\text{Lef}}_i$, resp. $\PP_T
V(\underline{d})^{\text{sm}}_i$,  
is the maximal open subset such that for every
$j=0,\dots,i$, 
\begin{enumerate}
\item[(i)] $\Xx_j\times_{\PP_T V(\underline{d})} \PP_T
  V(\underline{d})_i \to \PP_T V(\underline{d})_i$ 
is flat of relative
  dimension $n-j$,
\item[(ii)] the geometric fibers are reduced, and
\item[(iii)] every geometric fiber has, at worst, a single ordinary
  double point and no other singularities, resp. every geometric fiber
  is smooth.
\end{enumerate}
\end{defn}

\mni
By construction $\PP_T V(\underline{d})^{\text{sm}}_i$ is an open subset of
$\PP_T V(\underline{d})^{\text{Lef}}_i$.  

\begin{notat} \label{notat-CI2} \marpar{notat-CI2}
For each $i\geq 1$, denote by $\check{\Xx}_{i-1}$
the relative complement of $\PP_T V(\underline{d})^{\text{sm}}_{i}$ in
$\PP_T V(\underline{d})^{\text{Lef}}_{i-1}$.   
Denote by $F_{i-1}$
the relative complement of $\PP_T V(\underline{d})^{\text{Lef}}_{i}$
in $\PP_T V(\underline{d})^{\text{Lef}}_{i-1}$.   
\end{notat}

\mni
Note
that on $\PP_T V(\underline{d})^{\text{Lef}}_{i-1}$ 
there is a well-defined morphism $\Phi_{i-1}:\PP_T
V(\underline{d})^{\text{Lef}}_{i-1} 
\to \PP_T V(d_1,\dots,d_{i-1})$ that is flat.  In fact the image is the
corresponding open 
$$
\PP_T V(d_1,\dots,d_{i-1})^{\text{Lef}}_{i-1},
$$
and the morphism $\Phi_{i-1}$ to its image is Zariski locally on the
image isomorphic to the vector bundle $V(d_i,\dots,d_\cb)\times_T \PP_T
V(d_1,\dots,d_{i-1})^{\text{Lef}}_{i-1}$.  

\begin{cor} \label{cor-Lef} \marpar{cor-Lef}
If the characteristic is not $0$, assume that every $d_i\geq 2$.
With the same hypotheses as in Proposition \ref{prop-Lef},
for $i=1,\dots,\cb$, the closed subset $F_{i-1}$ has codimension $\geq 2$ in
$\PP_T V(\underline{d})^{\text{Lef}}_{i-1}$.  The
complement of $\PP_T V(\underline{d})^{\text{Lef}}_\cb$ in $\PP_T
V(\underline{d})$ has codimension $\geq 2$.  If $p$ has connected
geometric fibers and if $n\geq \cb+1$, resp. if $n\geq \cb+2$, 
then every geometric fiber of $\pr{2}$
is connected, 
resp. is normal and irreducible,
$$
\pr{2}:\Xx_\cb\times_{\PP_T V(\underline{d})} \PP_T
V(\underline{d})^{\text{Lef}}_\cb \to 
\PP_T V(\underline{d})^{\text{Lef}}_\cb.
$$ 
Finally, if $\text{char}(\FR)$ equals $0$, if $\Xx_T$ is smooth over
$\FR$, if $\Xx_T/T$ is acyclic, and if $n \geq \cb+ 4$,
then $\Xx_\cb\times_{\PP_T V(\underline{d})} \PP_T
V(\underline{d})^{\text{Lef}}_\cb$ is smooth over $\FR$, 
the morphism $\pr{2}$ above is acyclic, and the
natural map from $\text{Pic}_{\Xx_T/T}$ to the relative Picard scheme
of $\pr{2}$ is an isomorphism.
\end{cor}

\begin{proof}
The first assertion follows from Proposition \ref{prop-Lef} applied to
the restriction over $\PP_T V(\underline{d})^{\text{Lef}}_{i-1}$ of the morphism
$\Xx_{i-1}\to \PP_T V(\underline{d})$.   
Thus, by induction on $i$, for every $i=0,\dots,\cb$,
the closed complement of the open subset 
$\PP_T V(\underline{d})^{\text{Lef}}_i$ in $\PP_T
V(\underline{d})$ has codimension $\geq 2$.

\mni
Assuming that $n\geq \cb+1$,
connectedness of the fibers 
of $\Xx_i \times_{\PP_T V(\underline{d})} \PP_T
V(\underline{d})^{\text{Lef}}_i \to \PP_T
V(\underline{d})^{\text{Lef}}_i$ for $i=1,\dots,\cb$
is proved by induction on $i$ using
\cite[Corollaire 3.5, Expos\'{e} XII]{SGA2} for the induction step.
If $n\geq \cb+2$, then the geometric fibers of $\Xx_\cb$ are connected,
projective schemes of pure dimension $n-\cb\geq 2$ 
that are either smooth or else have a single
ordinary double point.  In particular, the geometric fiber is a local
complete intersection scheme that is regular away from codimension
$\geq 2$.  By Serre's Criterion, \cite[Th\'{e}or\`{e}me
5.8.6]{EGA4}, the geometric fiber is normal.  Since it is also
connected, it is irreducible.   

\mni
The acyclic hypothesis follows from Corollary
\ref{cor-iterate2} and induction on $\cb$.
\end{proof}

\bibliography{my}

\begin{thebibliography}{KMM87}

\bibitem[BGI71]{SGA6}
P.~Berthelot, A.~Grothendieck, and L.~Illusie, editors.
\newblock {\em Th\'eorie des intersections et th\'eor\`eme de
  {R}iemann-{R}och}.
\newblock Springer-Verlag, Berlin, 1971.
\newblock S\'eminaire de G\'eom\'etrie Alg\'ebrique du Bois-Marie 1966--1967
  (SGA 6), Dirig\'e par P. Berthelot, A. Grothendieck et L. Illusie. Avec la
  collaboration de D. Ferrand, J. P. Jouanolou, O. Jussila, S. Kleiman, M.
  Raynaud et J. P. Serre, Lecture Notes in Mathematics, Vol. 225.

\bibitem[dJHS11]{dJHS}
A.~J. de~Jong, Xuhua He, and Jason~Michael Starr.
\newblock Families of rationally simply connected varieties over surfaces and
  torsors for semisimple groups.
\newblock {\em Publ. Math. Inst. Hautes \'Etudes Sci.}, (114):1--85, 2011.

\bibitem[Ful84]{F}
William Fulton.
\newblock {\em Intersection theory}, volume~2 of {\em Ergebnisse der Mathematik
  und ihrer Grenzgebiete (3) [Results in Mathematics and Related Areas (3)]}.
\newblock Springer-Verlag, Berlin, 1984.

\bibitem[Gro62]{FGA}
Alexander Grothendieck.
\newblock {\em Fondements de la g\'eom\'etrie alg\'ebrique. [{E}xtraits du
  {S}\'eminaire {B}ourbaki, 1957--1962.]}.
\newblock Secr\'etariat math\'ematique, Paris, 1962.

\bibitem[Gro67]{EGA4}
A.~Grothendieck.
\newblock \'{E}l\'ements de g\'eom\'etrie alg\'ebrique. {IV}. \'{E}tude locale
  des sch\'emas et des morphismes de sch\'emas.
\newblock {\em Inst. Hautes \'Etudes Sci. Publ. Math. 20 (1964), 101-355; ibid.
  24 (1965), 5-231; ibid. 28 (1966), 5-255; ibid.}, (32):5--361, 1967.
\newblock \verb+http://www.numdam.org/item?id=PMIHES_1965__24__5_0+.

\bibitem[Gro68]{SGA2}
Alexander Grothendieck.
\newblock {\em Cohomologie locale des faisceaux coh\'erents et th\'eor\`emes de
  {L}efschetz locaux et globaux {$(SGA$} {$2)$}}.
\newblock North-Holland Publishing Co., Amsterdam, 1968.
\newblock Augment\'e d'un expos\'e par Mich\`ele Raynaud, S\'eminaire de
  G\'eom\'etrie Alg\'ebrique du Bois-Marie, 1962, Advanced Studies in Pure
  Mathematics, Vol. 2.

\bibitem[Gro03]{SGA1}
A.~Grothendieck.
\newblock {\em Rev\^etements \'etales et groupe fondamental ({SGA} 1)}.
\newblock Documents Math\'ematiques (Paris) [Mathematical Documents (Paris)],
  3. Soci\'et\'e Math\'ematique de France, Paris, 2003.
\newblock S\'eminaire de g\'eom\'etrie alg\'ebrique du Bois Marie 1960--61.
  [Geometric Algebra Seminar of Bois Marie 1960-61], Directed by A.
  Grothendieck, With two papers by M. Raynaud, Updated and annotated reprint of
  the 1971 original [Lecture Notes in Math., 224, Springer, Berlin; MR0354651
  (50 \#7129)].

\bibitem[Har77]{H}
Robin Hartshorne.
\newblock {\em Algebraic geometry}.
\newblock Springer-Verlag, New York, 1977.
\newblock Graduate Texts in Mathematics, No. 52.

\bibitem[KMM87]{KaMaMa}
Yujiro Kawamata, Katsumi Matsuda, and Kenji Matsuki.
\newblock Introduction to the minimal model problem.
\newblock In {\em Algebraic geometry, {S}endai, 1985}, volume~10 of {\em Adv.
  Stud. Pure Math.}, pages 283--360. North-Holland, Amsterdam, 1987.

\bibitem[SGA73]{SGA7II}
{\em Groupes de monodromie en g\'eom\'etrie alg\'ebrique. {II}}.
\newblock Lecture Notes in Mathematics, Vol. 340. Springer-Verlag, Berlin-New
  York, 1973.
\newblock S\'eminaire de G\'eom\'etrie Alg\'ebrique du Bois-Marie 1967--1969
  (SGA 7 II), Dirig\'e par P. Deligne et N. Katz.

\bibitem[Zhu]{Zhu}
Yi~Zhu.
\newblock Homogeneous space fibrations over surfaces.
\newblock to appear, J. Inst. Math. Jussieu.

\end{thebibliography}
\bibliographystyle{alpha}

\end{document}